\documentclass{amsproc}
\usepackage{eurosym}
\usepackage{amssymb}
\usepackage{amsfonts}
\usepackage{cmtt}
\usepackage{graphicx}					

\usepackage{calrsfs}
\DeclareMathAlphabet{\pazocal}{OMS}{zplm}{m}{n}

\DeclareMathOperator{\modul}{mod}
\DeclareMathOperator{\tw}{tw}			
\theoremstyle{plain}
\newtheorem{theorem}{Theorem}

\newtheorem{corollary}[theorem]{Corollary}
\newtheorem{definition}[theorem]{Definition}

\newtheorem{proposition}[theorem]{Proposition}

\newtheorem{remark}[theorem]{Remark}

\theoremstyle{remark}

\newcommand{\superscript}[1]{\ensuremath{^{\textrm{#1}}}}

\def\wu{\superscript{*}}
\def\wg{\superscript{$\star$}}

\begin{document}

\title[Modularity of minor-free graphs]{Modularity of minor-free graphs} 

\author[M.~Laso{\'n}]{Micha{\l} Laso{\'n}\wu\footnote{\wu michalason@gmail.com; Institute of Mathematics of the Polish Academy of Sciences, ul.\'{S}niadeckich 8, 00-656 Warszawa, Poland}} 

\author[M.~Sulkowska]{Ma{\l}gorzata Sulkowska\wg\footnote{\wg malgorzata.sulkowska@pwr.edu.pl; Universit{\'e} C{\^o}te d'Azur, CNRS, Inria, I3S, France; Wroc{\l}aw University of Science and Technology, Faculty of Fundamental Problems of Technology, Department of Fundamentals of Computer Science}}

\thanks{Micha{\l} Laso\'{n} was supported by Polish National Science Centre grant no. 2019/34/E/ST1/00087.}

\keywords{Modularity, minor-free graph, sparse graph, edge separator, Cheeger's inequality}

\begin{abstract} 
We prove that a class of graphs with an excluded minor and with the maximum degree sublinear in the number of edges is maximally modular, that is, modularity tends to $1$ as the number of edges tends to infinity. 
\end{abstract} 

\maketitle 

\section{Introduction}

\subsection{Modularity}

Modularity is a well-established parameter measuring the presence of community structure in the graph. It was introduced by Newman and Girvan in $2004$ (\cite{NeGi04}). Nowadays it is widely used as a quality function for community detection algorithms. The most popular heuristic clustering algorithms (check Louvain \cite{BlGuLaLe08} or Leiden \cite{TrWaEc19}) find the proper partition trying to maximize exactly this parameter. The definition of modularity is based on the comparison between the density of edges inside communities one observes in the graph and the density of edges one would expect if the edges of the graph were wired randomly preserving degree sequence. We make it precise just below.

Consider a simple undirected graph $G$ with $|V(G)|=n$ and $|E(G)|=m$. Whenever the context is clear we write  $V$ and $E$ for $V(G)$ and $E(G)$, respectively. For $v \in V$ by $\deg(v)$ denote the degree of a vertex $v$ in $G$. 
For a subset of vertices $S \subseteq V$ define $E(S)$ to be the set of edges in $G$ with both end-vertices within $S$ and let $\deg(S)=\sum_{v \in S}\deg(v)$. The \emph{modularity} of $G$ is defined as follows.

\begin{definition}[Modularity, \cite{NeGi04}] \label{def:modularity}
	Let $G$ be a graph with at least one edge. For a partition $\mathcal{A}$ of $G$ into induced subgraphs define its modularity score on $G$ as
	\[
	\modul_{\mathcal{A}}(G) = \sum_{A\in\mathcal{A}}\left(\frac{\vert E(A)\vert}{\vert E(G)\vert}-\left(\frac{\deg(V(A))}{\deg(V(G))}\right)^2\right).
	\]
	Modularity of $G$ is given by
	\[
	\modul(G) = \max_{\mathcal{A}}\modul_{\mathcal{A}}(G).
	\]
\end{definition}

Conventionally, a graph with no edges has modularity equal to $0$. For a given partition $\mathcal{A}$ the value $\sum_{A \in \mathcal{A}} \frac{|E(A)|}{|E(G)|}$ is called an \emph{edge contribution} while $\sum_{A \in \mathcal{A}} \left(\frac{\deg(V(A))}{\deg(V(G))}\right)^2$ is a \emph{degree tax}. A single summand of the modularity score is the difference between the fraction of edges within $A$ and the expected fraction of edges within $A$ if we considered a random multigraph on $V$ with the degree sequence given by $G$. 

It is easy to check that $0\leq \modul(G)<1$, and also that adding or deleting isolated vertices from the graph does not impact its modularity.

In practice, the problem of community detection very often concerns complex networks, i.e., graphs modeling real-life systems. Since most complex networks are sparse, it is natural to investigate which classes of sparse graphs exhibit high modularity. Our paper addresses exactly this question - modularity of commonly considered subclasses of nowhere dense graphs, that is a class of graphs introduced by Ne{\v{s}}et{\v {r}}il and Ossona de Mendez in \cite{NeOs11} as capturing the notion of sparsity.

\subsection{Related work} For a concise, up to date summary of modularity of various classes of graphs check the appendix of \cite{DiSk20} by McDiarmid and Skerman from $2020$. 

Here we focus on modularity of commonly considered subclasses of nowhere dense graphs. First, recall the definition of \emph{maximally modular} class of graphs.

\begin{definition} [Maximally modular class of graphs]
	A class of graphs $\mathcal{C}$ is maximally modular if for every $\varepsilon>0$ there exists $M_{\varepsilon}$ such that whenever $G$ is a graph from $\mathcal{C}$ with $m \geq M_{\varepsilon}$ edges, then $\modul(G)>1-\varepsilon$.
\end{definition}

It is not hard to show (consult \cite{MoSoVi11}) that a maximally modular class of graphs has the maximum degree sublinear in the number of edges. Hence, sublinear maximum degree is a necessary condition for a class of graphs to be maximally modular.

In $2018$ McDiarmid and Skerman formulated the following sufficient condition for a class of graphs to be maximally modular. By $\Delta(G)$ and $\tw(G)$ we denote respectively the maximum degree in $G$ and the treewidth of $G$.

\begin{corollary} [\cite{DiSk18}, Corollary $12$] \label{cor:McD_Sker}
	For $m = 1, 2, \ldots$ let $G_m$ be a graph with $m$ edges. If $tw(G_m) \cdot \Delta(G_m) = o(m)$ then  $\modul(G_m) \rightarrow 1$ as $m \rightarrow \infty$.
\end{corollary}

The above result is tight in a sense that $o(m)$ can not be replaced by $O(m)$. To justify it McDiarmid and Skerman present two examples. First, let $G$ be a star $K_{1,m}$ (with treewidth 1 and maximum degree $m$). Then, $\tw(G) \cdot \Delta(G)=m$ and, by \cite{MoSoVi11}, $\modul(G)=0$. Second, let $G$ be a random cubic graph on $n$ vertices (thus with $m=3n/2$ edges). Then, $\tw(G) \cdot \Delta(G)=O(m)$ and with high probability $\modul(G) \leq 0.79$ (see \cite{LiMi20,DiSk18}).

Thus, what follows from Corollary \ref{cor:McD_Sker} is that bounded treewidth in addition to the maximum degree sublinear in the number of edges already guarantees that the class of graphs is maximally modular. On the other hand the example of a random cubic graph gives that classes of bounded degree graphs are not maximally modular. These conclusions already lead to the classification presented in Figure \ref{fig:mod_sparse}.

\begin{figure}[!ht]
	\centering{
		\includegraphics[width = 0.7 \textwidth]{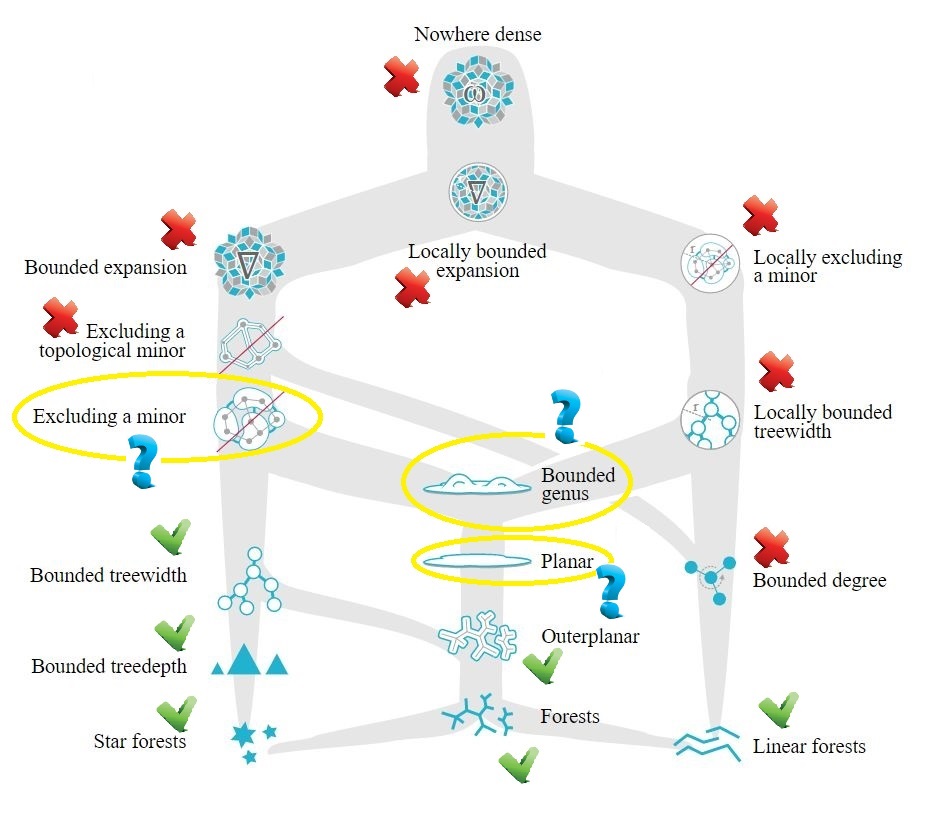}
		\caption{{\small Maximally and non-maximally modular subclasses of nowhere dense graphs. All classes of graphs are considered to have the maximum degree sublinear in the number of edges. (Background picture by Felix Reidl, source: https://tcs.rwth-aachen.de/$\sim$reidl/.)}} \label{fig:mod_sparse}
	}
\end{figure}

By Corollary \ref{cor:McD_Sker} one can obtain also some partial results for planar graphs and for bounded genus graphs.

Indeed, the random planar graph $G_n$ on $n$ vertices has $tw(G_n) = O(\sqrt{n})$ (which follows from the separator theorem for graphs of bounded genus \cite{GiHu84} and a recent result by Dvor{\'{a}}k and Norin \cite{DvNo19} establishing the linear dependence between the treewidth and the separation number for graphs of bounded genus). Next, with high probability $|E(G_n)| = \Theta(n)$ and, by \cite{DiRe08}, with high probability $\Delta(G_n) = O(\log{n})$. Thus, by Corollary \ref{cor:McD_Sker} random planar graphs are maximally modular.

Similarly, again by \cite{GiHu84,DvNo19}, for bounded genus graphs $\tw(G_m) = O(\sqrt{m})$. Thus, a class of bounded genus graphs with $\Delta(G_m) = o(\sqrt{m})$ is maximally modular.

\subsection{Our results} We prove the following.

\begin{theorem} \label{thm:main}
A class of graphs with an excluded minor and with the maximum degree sublinear in the number of edges is maximally modular.
\end{theorem}

Since classes of graphs with bounded genus and the class of planar graphs are subclasses of graphs with an excluded minor, the above theorem resolves, in positive, all three questions marked in Figure \ref{fig:mod_sparse} that remained unsolved. This way we achieve a complete classification of maximally modular classes among all commonly considered subclasses of nowhere dense graphs with sublinear maximum degree.
 
Our proof uses tools of spectral graph theory, in particular so-called Cheeger's Inequality and a recent important result by Biswal et al. (\cite{BiLeRa10}) for graphs with an excluded minor. 

\newpage
\section{Proof}

We begin by presenting tools of spectral graph theory and results for minor-free graphs that will be used in this section. 

Let $G=(V,E)$ be a simple undirected graph. Denote by $A$ its \emph{adjacency matrix} and by $D$ its \emph{diagonal degree matrix}. Recall that $L=D-A$ is the \emph{Laplacian matrix} of $G$. When $G$ has no isolated vertices, then $\mathcal{L}=D^{-\frac{1}{2}}LD^{-\frac{1}{2}}$ is the \emph{normalized Laplacian matrix} of $G$. Let $\lambda_i(L)$ and $\lambda_i(\mathcal{L})$ denote $i$\superscript{th} smallest eigenvalue of $L$ and $\mathcal{L}$, respectively. We concentrate on the second smallest eigenvalues, $\lambda_2(L)$ and $\lambda_2(\mathcal{L})$, as they carry information about the connectivity of $G$.


We will need to justify the existence of sufficiently small and reasonably balanced edge cuts in considered graphs. This will be done by Cheeger's inequality, which was first established for manifolds in $1970$ by Cheeger \cite{Ch70}, while the graph version is due to Alon and Milman \cite{Al86,AlMi85}.

\begin{theorem}[Cheeger's Inequality, \cite{AlMi85}] \label{TheoremCheeger} 
	When $G$ has no isolated vertices, then
	$$\frac{\lambda_2(\mathcal{L})}{2}\leq\min_{\emptyset\neq S\subsetneq V}\frac{|E(S,V\setminus S)|}{\min\{\deg(S),\deg(V\setminus S)\}}\leq \sqrt{2\lambda_2(\mathcal{L})},$$ where $E(S,V\setminus S)$ denotes the set of edges in $G$ between $S$ and $V\setminus S$.
\end{theorem}

One of the first results for minor-free graphs shows that they have at most a linear number of edges in terms of the number of vertices.

\begin{theorem}[\cite{Ma67}]\label{TheoremSparse}
	A class of graphs with an excluded minor has at most a linear number of edges. That is, for every graph $H$ there exists a constant $c_H$ such that every graph on $n$ vertices without a minor $H$ has at most $c_Hn$ edges.
\end{theorem}

A recent important result by Biswal et al. for minor-free graphs gives an upper bound for the second smallest eigenvalue of the Laplacian matrix. 

\begin{theorem}[\cite{BiLeRa10}, Theorem $5.3$]\label{TheoremMinorFree}
	If $G$ is $K_h$-minor-free and $|V|\geq c_1h^2\log h$, then $\lambda_2(L)\leq c_2\frac{\Delta(G)h^6\log h}{|V|}$ for some positive absolute constants $c_1,c_2$.
\end{theorem}

\begin{remark}\label{RemarkLambda}
	Notice that $\lambda_2(\mathcal{L})\leq\lambda_2(L)$ for graphs without isolated vertices. 
\end{remark}

\begin{proof}
	We have $\mathcal{L}=D^{-\frac{1}{2}}LD^{-\frac{1}{2}}=(D^{-\frac{1}{2}}LD^{\frac{1}{2}})D^{-1}=:L'D^{-1}$. Matrices $L$ and $L'$ are similar, hence have the same spectrum. Recall that $\lambda_1(\mathcal{L})=\lambda_1(L)=0$, so also $\lambda_1(L')=0$. Now, we restrict linear maps $\mathcal{L}$ and $L'$ to the subspace, denoted by $\ker^{\bot}$, orthogonal to the kernel of $\mathcal{L}$ and $L'$. Since $D^{-1}$ is a shrinking linear map (module of every eigenvalue less or equal to $1$) we get that the module of the smallest eigenvalue of $\mathcal{L}|_{\ker^{\bot}}$ is less or equal to  the module of the smallest eigenvalue of $L'|_{\ker^{\bot}}$. That is, since these eigenvalues are real and nonnegative, $\lambda_2(\mathcal{L})\leq \lambda_2(L')=\lambda_2(L)$.	
\end{proof}

\begin{proposition}\label{Proposition1}
	A class of graphs with an excluded minor, the maximum degree sublinear in the number of edges, and vertices weighted proportionally to their degree has a sublinear weighted edge separator. 
	
	That is, for every graph $H$ and $\varepsilon>0$ there exists $\delta=\delta(\varepsilon)>0$ such that the following is true. Let $G$ be a graph without minor $H$, with $m$ edges, and maximum degree at most $\delta m$. A weight $w(v)=\frac{\deg(v)}{\deg(V)}$ is assigned to every vertex $v$ of $G$. Then, there is a set of no more than $\varepsilon m$ edges in $G$ whose deletion creates a graph in which the total weight of every connected component is smaller than $\varepsilon$.
\end{proposition}

\begin{proof}
	Fix a graph $H$ and $\varepsilon>0$, and let $h=|H|$. Now, choose $\delta>0$ such that $\frac{\varepsilon}{2}\frac{1}{\delta}\frac{1}{c_H}\geq c_1h^2\log h$ and $(\lfloor\log_{2}\frac{1}{\varepsilon}\rfloor+2)2\sqrt{2c_2c_H\frac{2}{\varepsilon}\delta h^6\log h}<\varepsilon$. 
	
	Suppose $G(V,E)$ is a graph without minor $H$, with $m$ edges, and maximum degree at most $\delta m$. In particular, $m\geq\frac{1}{\delta}$. We will show that the statement of the proposition holds for $G$, $\varepsilon$, and $\delta$. Notice that without loss of generality we may assume that $G$ does not have isolated vertices, as these do not impact weights. 
	
	Consider the following procedure which takes on input an induced subgraph $G'(V',E')$ of $G$ with
	$2|E'|=\deg_{G'}(V')\geq\frac{\varepsilon}{2}\deg(V)=\varepsilon m$. Notice that by Theorem \ref{TheoremSparse} we have $|V'|=|E'|\frac{|V'|}{|E'|}\geq\frac{\varepsilon}{2}m\cdot\frac{1}{c_H}\geq\frac{\varepsilon}{2}\frac{1}{\delta}\frac{1}{c_H}\geq c_1h^2\log h$. Clearly, $G'$ is $H$-minor-free, hence also $K_h$-minor-free, thus by Theorem \ref{TheoremMinorFree} since $|V'|\geq c_1h^2\log h$, we have that $\lambda_2(L')\leq c_2\frac{\delta m\cdot h^6\log h}{|V'|}=c_2\frac{m}{|V'|}\delta h^6\log h\leq c_2\frac{|E'|}{|V'|}\frac{2}{\varepsilon}\delta h^6\log h\leq c_2c_H\frac{2}{\varepsilon}\delta h^6\log h$. Thus, by Theorem \ref{TheoremCheeger} and Remark \ref{RemarkLambda}, we get that for some proper subset $S$ of $V'$:
	\begin{equation}\label{cut}
	|E'(S,V'\setminus S)|\leq\min\{\deg_{G'}(S),\deg_{G'}(V'\setminus S)\}\sqrt{2c_2c_H\frac{2}{\varepsilon}\delta h^6\log h}.
	\end{equation}
	The procedure deletes edges $E'(S,V'\setminus S)$ from $G'$, and returns as output all connected components of a resulting graph. These are also induced subgraphs of $G$.
	
	Now, we begin a process starting with the set $T$ consisting of the graph $G$ and repeatedly apply the above procedure to elements of $T$ (that is, induced subgraphs $G'$ of $G$) satisfying $\deg_{G'}(V')\geq\frac{\varepsilon}{2}\deg(V)$. 
	
	Notice that since the degree of each element of the output is smaller than the degree of the input, the process has to end. Keep in mind that at any step of the process vertices of elements of $T$ form a partition of the set $V$. Moreover, after first step of the process elements of $T$ are connected induced subgraphs of $G$.
	
	Suppose that at the end of the process the set $T$ consists of connected induced subgraphs $G_1,\dots,G_t$. Denote by $D$ the set of edges deleted during the process. 
	
	In order to count how many edges were deleted during the process in total, assign in a single procedure that started with $G'$ edges $E'(S,V'\setminus S)$ to vertices in $S$ proportionally to their $G'$-degree when $\deg_{G'}(S)\leq\deg_{G'}(V'\setminus S)$ and to $V'\setminus S$ otherwise. Now, by the inequality (\ref{cut}) in this single run of the procedure
	every vertex $v$ in $S$ got assigned at most $\deg_{G'}(v)\sqrt{2c_2c_H\frac{2}{\varepsilon}\delta h^6\log h}$ deleted edges when $\deg_{G'}(S)\leq\deg_{G'}(V'\setminus S)$ and $0$ otherwise. 
	
	Notice that every single vertex $v$ gets assigned nonzero deleted edges at most $\lfloor\log_{\frac{1}{2}}\frac{\varepsilon}{2}\rfloor+1=\lfloor\log_{2}\frac{1}{\varepsilon}\rfloor+2$ times. Indeed, it happens when the degree of the induced subgraph to which $v$ belongs gets at least halved. It starts from $\deg(V)$ and ends just after dropping below $\frac{\varepsilon}{2}\deg(V)$. Therefore, summing over all vertices, the total number of deleted edges $|D|$ is at most 
	$(\lfloor\log_{2}\frac{1}{\varepsilon}\rfloor+2)2m\sqrt{2c_2c_H\frac{2}{\varepsilon}\delta h^6\log h}<\varepsilon m$.
	
	Now, connected components of the graph $G\setminus D$ are $G_i$'s. Every $G_i$ has weight equal to $\frac{1}{deg(V)}\deg_{G}(V(G_i))\leq\frac{1}{deg(V)}(\deg_{G_i}(V(G_i))+|D|)<\frac{\varepsilon}{2}+\frac{\varepsilon}{2}=\varepsilon$. 
\end{proof}

\begin{remark}
	Alon, Seymour, and Thomas \cite{AlSeTh90} generalized the planar \emph{vertex} separator theorem of Lipton and Tarjan \cite{LiTa79} to minor-free graphs.
	
	It is known \textnormal{(}see \cite{DiDjSyVr93} using a result of \cite{Mi86}\textnormal{)} that a planar graph $G$ has an \emph{edge} separator of size $O(\sqrt{\Delta(G)|V(G)|})$. It would be interesting to generalize this theorem to minor-free graphs, as this would strengthen Proposition \ref{Proposition1}.
\end{remark}

{
	\renewcommand{\thetheorem}{\ref{thm:main}}
	\begin{theorem} 
		Let $\mathcal{C}$ be a class of graphs excluding a fixed minor $H$ and with the maximum degree sublinear in the number of edges -- that is, for every $\delta>0$ there exists $m_\delta$ such that if  $m \geq m_\delta$ and $G$ is a graph from $\mathcal{C}$ with $m$ edges, then $\Delta(G) \leq \delta m$. Then, for every $\varepsilon>0$ there exists $M_{\varepsilon}$ such that if  $m \geq M_{\varepsilon}$ and $G$ is a graph from $\mathcal{C}$ with $m$ edges, then $\modul(G) \geq 1 - \varepsilon$.
	\end{theorem}
	\addtocounter{theorem}{-1}
}

\begin{proof}		
	Fix a class of graphs $\mathcal{C}$ as in the assumption and fix $\varepsilon>0$. Now, choose $\delta:=\delta(\frac{\varepsilon}{2})$ such that the assertion of Proposition \ref{Proposition1} holds. We will show that if $m\geq M_{\varepsilon}:=m_{\delta}$ and $G$ is a graph from $\mathcal{C}$ with $m$ edges, then $\modul(G)>1-\varepsilon$.
	
	Indeed, then by Proposition \ref{Proposition1} there is a set of no more than $\frac{\varepsilon}{2}m$ edges in $G$ whose deletion creates a graph in which the total weight of every connected component is less than $\frac{\varepsilon}{2}$. Now, let $\mathcal{A}$ be the set of those connected components and let $D$ be the set of deleted edges. 
		
	Firstly, notice that 
	$$\sum_{A\in\mathcal{A}}\frac{\vert E(A)\vert}{\vert E(G)\vert}=\frac{\vert E(G)\vert-\vert D\vert}{\vert E(G)\vert}\geq 1-\frac{\varepsilon}{2}.$$
	
	Secondly, we have
	$$\sum_{A\in\mathcal{A}}\left(\frac{\deg(V(A))}{\deg(V(G))}\right)^2=\sum_{A\in\mathcal{A}}w(A)^2<\sum_{A\in\mathcal{A}}\frac{\varepsilon}{2}w(A)=\frac{\varepsilon}{2}.$$
	
	Concluding,
	$$\modul(G)\geq\sum_{A\in\mathcal{A}}\left(\frac{\vert E(A)\vert}{\vert E(G)\vert}-\left(\frac{\deg(V(A))}{\deg(V(G))}\right)^2\right)>1-\varepsilon.$$
\end{proof}

\bibliographystyle{plain}
\bibliography{modularity}




\end{document}